\documentclass[a4paper,12pt,reqno]{amsart}
\usepackage{amssymb,amsthm}
\usepackage{ifthen}
\usepackage{graphicx}
\usepackage{float}
\usepackage{tcolorbox}
\usepackage{diagbox}
\theoremstyle{plain}

\newtheorem{thm}{Theorem}[section]
\newtheorem*{thm*}{Theorem}

\numberwithin{equation}{section}

\newtheorem{cor}{Corollary}[section]

\newtheorem{prop}{Proposition}[section]

\newtheorem{defn}{Definition}[section]

\newtheorem{conj}{Conjecture}

\theoremstyle{definition}
\newcounter {own}
\def\theown {\thesection  .\arabic{own}}

\newenvironment{pf}[1][]{%
 \vskip 3mm
 \noindent
 \ifthenelse{\equal{#1}{}}%
  {{\slshape Proof. }}%
  {{\slshape #1.} }%
 }%
{\qed\bigskip}

\newcounter{alphabet}
\newcounter{tmp}

\newcommand{\ds}{\displaystyle}

\newcounter{minutes}
\divide\time by 60
\newcounter{hours}
\multiply\time by 60 \addtocounter{minutes}{-\time}

\begin{document}
\bibliographystyle{amsplain}
\title{A new sampling density condition for shift-invariant spaces}

\thanks{
File:~\jobname .tex,
          printed: 2013-07-30,
          \thehours.\ifnum\theminutes<10{0}\fi\theminutes}

\author{A. Antony Selvan}

\address{A. Antony Selvan, Post Doctoral Fellow,
The Institute of Mathematical Sciences, Chennai--600 113, India.}
\email{antonyaans@gmail.com}
%
\maketitle
\pagestyle{myheadings}
\markboth{A. Antony Selvan}{A new sampling density condition for shift-invariant spaces}

\begin{abstract}
Let $X=\{x_i:i\in\mathbb{Z}\}$, $\dots<x_{i-1}<x_i<x_{i+1}<\dots$, be a sampling set which is separated by a constant $\gamma>0$. Under certain conditions on $\phi$, it is proved that 
if there exists a positive integer $\nu$ such that 
$$\delta_\nu:=\sup\limits_{i\in\mathbb{Z}}(x_{i+\nu}-x_i)<\dfrac{\nu}{2\pi}\left(\dfrac{c_{k}^2}{M_{2k}}\right)^{\frac{1}{4k}},$$ then 
every function belonging to a shift-invariant space $V(\phi)$  can be reconstructed stably from its nonuniform sample values  $\{f^{(j)}(x_i):j=0,1,\dots, k-1, i\in\mathbb{Z}\}$,
where $c_k$ is a Wirtinger-Sobolev constant and $M_{2k}$ is a constant in Bernstein-type  inequality of $V(\phi)$. Further, when $k=1$, the maximum gap 
$\delta_\nu<\nu$ is sharp for certain shift-invariant spaces.
\vspace{5mm} \\
\noindent {\it Key words and phrases} : bandlimited functions, Bernstein's inequality, frames, Hermite interpolation, nonuniform sampling, shift-invariant spaces, Wirtinger-Sobolev inequality.
\vspace{3mm}\\
\noindent {\it 2000 AMS Mathematics Subject Classification} ---42C15, 94A20
\end{abstract}

%
%
\section{Introduction}
Let $\mathcal{B}_\sigma$ denote the space of all $\sigma$-bandlimited functions, \textit{i.e.,}
\begin{eqnarray*}
\mathcal{B}_\sigma=\left\{f\in L^2(\mathbb{R}): \textrm{supp}\widehat{f}\subseteq [-\tfrac{\sigma}{2\pi},\tfrac{\sigma}{2\pi}]\right\},
\end{eqnarray*}
where $\widehat{f}$ denotes the Fourier transform of $f$, defined by $\widehat{f}(w)=\int_{-\infty}^\infty f(x)e^{-2\pi ixw}\mathrm{d}x$. The celebrated theorem of Paley-Wiener says that $\mathcal{B}_\sigma$ coincides with the space of
entire functions  of exponential type $\leq\sigma$. It is well known that $\mathcal{B}_\sigma$ is a reproducing kernel Hilbert space with reproducing kernel $\mathcal{K}(x,y)=\tfrac{\sin\sigma\left(x-y\right)}{\sigma \left(x-y\right)}$. 
The classical Shannon's sampling theorem states that every $f\in\mathcal{B}_\sigma$ can be reconstructed from the sampling formula
\begin{eqnarray*}
f(x)=\ds\sum_{k\in\mathbb{Z}}f\left(\dfrac{k\pi}{\sigma}\right)\dfrac{\sin \sigma\left(x-k\pi/\sigma\right)}{\sigma \left(x-k\pi/\sigma\right)}.
\end{eqnarray*}

The problem of uniform sampling is well studied in the literature (\cite{Higg}, \cite{Jerri}, \cite{Papoulis}). In \cite{PaWi}, Paley and Wiener  extended the Shannon's sampling theorem to nonuniform sampling set.
They showed that if $X=\{x_k\in\mathbb{R}:k\in\mathbb{Z}\}$ is such that $|x_k-k|<\tfrac{1}{\pi^2}$, $k\in\mathbb{Z}$, then any $f\in\mathcal{B}_{\pi}$ can be recovered from its sample values $\{f(x_k)$: $k\in\mathbb{Z}\}$.
Later, Kadec in \cite{Kadec} showed that the maximum bound for $|x_k-k|$ has to be less than 0.25.
The general nonuniform sampling theory of bandlimited function is completely understood by the works of Beurling, Landau, and others. 
In fact, if the Beurling density, $D(X)$ of the set $X$, is greater than $\tfrac{\pi}{\sigma}$, then any $\sigma$-bandlimited function $f$ can be reconstructed stably from its sample
values $\{f(x_i) : x_i \in X\}$. We refer to \cite{Beurling}, \cite{Landau} and \cite{Ortega} in this connection.

It is difficult to compute the Beurling density $D(X)$ in general. Instead of Beurling density, Gr\"{o}chenig  in \cite{Grochenig} gave an easy sufficient condition for stable set of sampling for $\mathcal{B}_\sigma$.
In that paper, he proved that if $\sup_{i}(x_{i+1}-x_i)<\tfrac{\pi}{\sigma}$, then any bandlimited function  can be reconstructed from its nonuniform samples $\{f(x_i) :i\in\mathbb{Z}\}$.
In \cite{Raza}, Razafinjatovo obtained a frame algorithm for reconstructing a function $f\in \mathcal{B}_\sigma$ from its nonuniform samples $\{f^{(j)}(x_i):j=0,1,\dots, k-1, i\in\mathbb{Z}\}$ 
with maximum gap condition, namely  $\sup_{i}(x_{i+1}-x_i)=\delta<\frac{2}{\sigma}((k-1)!\sqrt{(2k-1)2k})^{1/k}$ using Taylor's polynomial approximation. Recently, Radha and the author in \cite{AntoRad2}
improved the  maximum gap condition using Hermite interpolation.

Shift-invariant spaces serve as a universal model for sampling problem as it includes a large class of functions whether bandlimited or not by appropriately choosing a generator. 
Sampling in shift-invariant spaces that are not bandlimited is a suitable and realistic model for many applications, e.g., for taking into account real acquisition and
reconstruction devices, for modeling signals with smoother spectrum than is the case
with bandlimited functions, or for numerical implementation (see \cite{AlUnser}, \cite{FeiGr} and \cite{UnserAl}). 
The problem of non-uniform sampling in general shift-invariant spaces was studied by Aldroubi and Gr\"{o}chenig in \cite{AlGr2}.

The fundamental problem of sampling in shift-invariant spaces is to find a characterization of stable set of sampling in shift-invariant spaces.
Sampling theorems in shift-invariant spaces are of fundamental importance in digital signal processing since it provides a means of converting an analogue signal into a digital signal from its sample values.
There are many qualitative sampling theorems for shift-invariant spaces. They state that for every reasonable generator $\phi$, there  exists  a $\delta>0$
such  that  every  set  with  maximum gap $\sup_{i}(x_{i+1}-x_i)<\delta$ is a stable set of sampling for $V(\phi)$, 
\textit{i.e.}, there exist $r$, $R>0$ such that
\begin{equation*}
r\|f\|^2\leq\sum_{i\in\mathbb{Z}}|f(x_{i})|^2\leq R\|f\|^2,~\textrm{~for~all~} f\in V(\phi).
\end{equation*}
The only generators for which the sharp result is known are $B$-splines, ripplets, and
totally positive functions of finite order $\geq 2$ (see \cite{AlGr1}, \cite{Grochenig3}). 
For other shift-invariant spaces, the problem of finding the sharp maximum gap condition is still open. 
Recently, Radha and the author in \cite{AntoRad3} found the sharp maximum gap condition for a shift-invariant space generated by a Meyer scaling function.

It is well-known that the best choice  of interpolation points to minimize  
the  error  estimate of the  Lagrange interpolation on  $[-1,1]$ corresponds to the  roots of the Chebyshev 
polynomials  of the first kind. In the case of an arbitrary interval $[a,b]$, the points
$$x_i=\dfrac{b-a}{2}\cos\dfrac{(2i+1)\pi}{2n+2}+\dfrac{b+a}{2},~~i=0,1,\dots,n,$$
are the best choice  of interpolation points to minimize  
the  error  estimate of the  Lagrange interpolation on this interval (see \cite{Davis}).
It is important to study the behaviors of sampling sets which  are chosen with respect to Chebyshev polynomial in a certain way. 
Although the maximum gap condition $\sup_{i}(x_{i+1}-x_i)<1$ is easy to verify, it does not include the Chebyshev nonuniform samples in general. 
For example, consider the  Chebyshev sampling set $X$  whose elements are chosen in this way:
$$x_i=1.5\cos\dfrac{(2i+1)\pi}{8}+1.5,~~i=0,1,2,3,~\text{and}$$
$$x_{4n+j}=nx_3+x_i,~n\in\mathbb{Z}, ~~i=0,1,2,3.$$
Since the set $X$ does not satisfy the property $\sup_{i}(x_{i+1}-x_i)<1$, we cannot use Gr\"{o}chenig's result in \cite{Grochenig}.
Even if we change the distance between atleast two consecutive points to be greater than $1$, Gr\"{o}chenig's result is not applicable.
However, there are sampling sets so that $\sup_{i}(x_{i+1}-x_i)>1$ but there exists $\nu\in\mathbb{N}$ such that $\sup_{i}(x_{i+\nu}-x_i)<a_\nu$, for some $a_\nu>0$.
It is easy to check that the above sampling set satisfies  $\sup_{i}(x_{i+3}-x_i)<3$. From the above discussion, it seems natural to ask the following question: 
\begin{quotation}
\begin{tcolorbox}
\textit{For every reasonable generator $\phi$, does there  exist  a constant  $a_\nu>0$
such  that  every  set  with  maximum gap $$\sup\limits_{i\in\mathbb{Z}}(x_{i+\nu}-x_i)<a_\nu,~~\text{for some}~\nu\in\mathbb{N},$$ is a stable set of sampling for $V(\phi)$?}
\end{tcolorbox}
\end{quotation}

\begin{defn}
For a fixed positive integer $\nu$, a sequence $X=\{x_i:{i\in\mathbb{Z}}\}$, $\dots<x_{i-1}<x_i<x_{i+1}<\cdots,$ is called $\delta_\nu$-dense if $\sup\limits_{i\in\mathbb{Z}}(x_{i+\nu}-x_i)=\delta_\nu.$
\end{defn}
In this paper, we answer the above question for certain shift-invariant spaces. In particular, we prove the following
\begin{thm}\label{pap7thm1.1}
Assume that $\{x_i:{i\in\mathbb{Z}}\}$ is separated by a constant $\gamma>0$. If
\begin{eqnarray}\label{pap7eqn1.1}
 \sup\limits_{i\in\mathbb{Z}}(x_{i+\nu}-x_i)<\nu,~~\text{for some}~\nu\in\mathbb{N},  
\end{eqnarray}
then there exist positive constants $A_\nu$ and $B_\nu$ such that
\begin{eqnarray*}\label{pap3eqn2.7}
A_\nu\| f\|_2^2&\leq&\ds\sum\limits_{i\in\mathbb{Z}}|f(x_i)|^2\leq B_\nu\| f\|_2^2, ~\text{for all}~ f\in V(\phi),
\end{eqnarray*}
where $\phi(x)=\dfrac{\sin \pi x}{\pi x}$ or Meyer scaling function. Further, the maximum gap condition \eqref{pap7eqn1.1} is sharp.
\end{thm}
The paper is organized as follows. In section 2, we discuss some useful facts in shift-invariant spaces, Wirtinger-Sobolev inequality and Hermite interpolation. In section 3, we first prove Bernstein inequality for
shift-invariant space $V(\phi)$ and provide a sufficient condition for stable set of sampling in $V(\phi)$. In section 4, we mention some open problems.
\section{Preliminaries}
This section provides some useful terminology and results in order to prove our main results.
\begin{defn}
A sequence of vectors $\{f_n:~n\in\mathbb{Z}\}$ in  a separable
Hilbert space $\mathcal{H}$ is said to be a {\it frame} if there
exist constants $0< m \leq M<\infty$ such that
\begin{equation}
m\|f\|_\mathcal{H}^2\leq\displaystyle\sum_{n\in\mathbb{Z}}|\langle
f,f_n\rangle_\mathcal{H}|^2\leq M\|f\|_\mathcal{H}^2,
\end{equation}
for every $f\in\mathcal{H}$.
\end{defn}
\begin{defn}
A sequence of vectors $\{f_n:~n\in\mathbb{Z}\}$ in  a separable
Hilbert space $\mathcal{H}$ is said to be a {\it Riesz basis} if
$\overline{span\{f_n\}}=\mathcal{H}$ and there exist constants $0< c
\leq C<\infty$ such that
\begin{equation}
c\sum_{n\in\mathbb{Z}}|d_n|^2\leq\big\|\sum_{n\in\mathbb{Z}}d_nf_n\big\|^2_{\mathcal{H}}\leq
C\sum_{n\in\mathbb{Z}}|d_n|^2,
\end{equation}
for all $(d_n)\in\ell^2(\mathbb{Z})$. 
\end{defn}
\begin{defn}
A closed subspace $M$ of $L^2(\mathbb{R})$ is called a {\it shift-invariant space} if $T_{n}\phi\in M$, for every $\phi\in M$ and
$n\in\mathbb{Z}$, where $T_n$ is the translation operator
defined by $T_n\phi(x)=\phi(x-n)$, for all $x\in\mathbb{R}$.
For $\phi\in L^2(\mathbb{R})$,
$\overline{span\{T_n\phi:n\in\mathbb{Z}\}}$ is called shift-invariant space generated by $\phi$ and is denoted by
$V(\phi)$.
\end{defn}
Every Riesz basis is a frame. It is well known that $\{T_n\phi:n\in\mathbb{Z}\}$ is a Riesz
basis for $V(\phi)$ if and only if
\begin{equation}\label{pap4eqn1.3}
 0<c\leq G_\phi(w)\leq C<\infty \hspace{1.5cm}
 a.e.~~w\in\mathbb{R},
\end{equation}
where
$G_\phi(w):=\displaystyle\sum_{n\in\mathbb{Z}}|\widehat{\phi}(w+n)|^2$.
Moreover, $\|G_\phi\|_0:=\mathop{\rm{essinf}}
\limits_{w\in[0,1]} G_\phi(w)$ and
$\|G_\phi\|_{\infty}:=\mathop{\rm{esssup}}\limits_{w\in[0,1]}
G_\phi(w)$ are the optimal Riesz bounds for $\{T_n\phi:n\in\mathbb{Z}\}$.
Recall that a closed subspace $V$ of
$L^2(\mathbb{R})$ is said to be a reproducing kernel Hilbert space if there exists a function $\mathcal{K}_x\in V$ such that
\begin{equation}
f(x)=\langle f,\mathcal{K}_x\rangle,~ \textrm{for every}~ f\in V.
\end{equation}
The function $\mathcal{K}(x,y):=\mathcal{K}_x(y)=\langle \mathcal{K}_x,\mathcal{K}_y\rangle$ is called the
reproducing kernel of $V$. If $\{\phi_n\}$ is a Riesz basis for $V$, then
\begin{equation}
\mathcal{K}(x,y)=\displaystyle\sum_n\overline{\phi_n(x)}\widetilde{\phi}_n(y)
\end{equation}
is the reproducing kernel for $V$, where $\{\widetilde{\phi}_n\}$ is
the dual basis of $\{\phi_n\}$. If $\phi$ is a continuous function having compact support, then $V(\phi)$ is a reproducing kernel Hilbert space.
For a detailed study of sampling and reconstruction in shift-invariant spaces, we refer to \cite{AlGr2}.
\begin{defn}
Let $\{x_{i}:i\in\mathbb{Z}\}$ be a sequence of real or complex numbers. Then
\begin{enumerate}
\item [$(i)$]$\{x_{i}:i\in\mathbb{Z}\}$ is separated by a constant $\gamma>0$ if $\inf\limits_{i\neq j} |x_i-x_j|\geq\gamma$.
\item [$(ii)$]$\{x_{i}:i\in\mathbb{Z}\}$ is said to be a  stable set of sampling for the reproducing kernel Hilbert space $\mathcal{H}$ 
if there exist constants $r$, $R>0$ such that
\begin{equation*}
r\|f\|^2\leq\sum_{i\in\mathbb{Z}}|f(x_{i})|^2\leq R\|f\|^2,~\textrm{~for~all~} f\in \mathcal{H}.
\end{equation*}
\end{enumerate}
\end{defn}
\begin{defn}
The Wiener amalgam space $W(C,\ell^1)$ is defined as
\begin{equation*}
W(C,\ell^1):=\left\{f\in C(\mathbb{R}):\|f\|_W:=\sum_{n\in\mathbb{Z}}
\max\limits_{x\in[0,1]}|f(x+n)|<\infty\right\}.
\end{equation*}
\end{defn}
Wiener amalgam spaces have been introduced by Feichtinger in \cite{Feich}. The basic idea behind these spaces is to generate spaces of functions or distributions which show a local behaviour (local membership in a certain function space) and 
a certain global behaviour (expressed in terms of local norm). If $\phi\in W(C,\ell^1)$, then $V(\phi)$ is a reproducing kernel Hilbert space (see \cite{AlGr1}).
\begin{defn}
Let $g\in L^2(\mathbb{R})$ be a non-zero window function. Then
the windowed Fourier transform of $f\in L^2(\mathbb{R})$  is defined as
\begin{eqnarray}\label{STFT}
F_gf(x,y):=\int\limits_{-\infty}^{\infty}f(t)\overline{g(t-x)}e^{-2\pi ity}~\mathrm{d}t.
\end{eqnarray}
\end{defn}
The windowed Fourier transform was developed to obtain a frequency
analysis which is also local in time for a given signal. 
The connection between windowed Fourier transform and sampling in shift-invariant spaces
is given in \cite{Grochenig3}.

The following three results are the main ingredients to prove our main results.

\begin{thm}[Wirtinger-Sobolev inequality]\label{pap2WirSob}$(\mathop{\rm cf.~}$\cite{BoWi}$)$.
Let $f$ be a complex valued  function defined on the interval $\left[a,b\right]$.
If $f\in C^{r}\left[a,b\right]$ with $f^{(l)}(a)=f^{(l)}(b)=0$, $0\leq l\leq r-1$, then
\begin{eqnarray}\label{pap7eqn2.6}
\int\limits_a^b|f(x)|^2~\mathrm{d}x\leq \dfrac{(b-a)^{2r}}{c_r}\int\limits_a^b|f^{(r)}(x)|^2~\mathrm{d}x,
\end{eqnarray}
where $c_r$ is the minimal eigenvalue of the boundary value problem
\begin{eqnarray*}
(-1)^r u^{(2r)}(x)=\lambda u(x),~~ x\in[0,1],\\
u^{(k)}(0)=u^{(k)}(1)=0, ~~0\leq k\leq r-1,~~u\in C^{2r}[0,1].
\end{eqnarray*}
\end{thm}
It is proved in \cite{AntoRad3} that Wirtinger-Sobolev inequality is still true if the right hand side of \eqref{pap7eqn2.6} is replaced by
\begin{eqnarray}\label{pap7eqn2.6a}
\dfrac{1}{c_r^2}(b-a)^{4r}\int\limits_a^b|f^{(2r)}(x)|^2~\mathrm{d}x,
\end{eqnarray}
$f\in C^{2r}\left[a,b\right]$ with $f^{(l)}(a)=f^{(l)}(b)=0$, $0\leq l\leq r-1$.
We explicitly mention the values and bounds for the constants $c_r$ as given in \cite{BoWi}.
\begin{eqnarray*}
c_1=\pi^2, c_2=500.5467, c_3=61529.
\end{eqnarray*}
For any $r\geq 1$, the Wirtinger-Sobolev constant $c_k$ satisfies the following property:
\begin{eqnarray}\label{pap3eqn2.18}
\dfrac{4r-2}{4r^2-r}\dfrac{(4r)!(r!)^2}{(2r!)^2}\leq c_r\leq\dfrac{4r+1}{2r+1}\dfrac{(4r)!(r!)^2}{(2r!)^2}.
\end{eqnarray}
Further, as $r\to \infty$, $c_r=\sqrt{8\pi r}\left(\dfrac{4r}{e}\right)^{2r}\left[1+\mathcal{O}\Big(\dfrac{1}{\sqrt{r}}\Big)\right]$.
\begin{prop}$(\mathop{\rm cf.~}$\cite{FeiGr},\cite{Raza}$)$.\label{pap2prop2.1}
Let $A$ be a bounded operator on a Hilbert space $\mathcal{H}$ that satisfies
\begin{align*}
\|f-Af\|_\mathcal{H}\leq C\|f\|_\mathcal{H},
\end{align*}
for every $f\in \mathcal{H}$ and for some $C$, $0<C<1$.
Then $A$ is invertible on $\mathcal{H}$ and $f$ can be recovered from $Af$ by the following iteration algorithm.
Setting
\begin{eqnarray*}
f_0&=& A f ~and \\
f_{n+1}&=&f_n+A(f-f_n), ~n\geq 0,
\end{eqnarray*}
we have $\lim\limits_{n\to\infty}f_n=f$. The error estimate after $n$ iterations is
\begin{align*}
\|f-f_n\|_\mathcal{H}\leq C^{n+1}\|f\|_\mathcal{H}.
\end{align*}
\end{prop}
\begin{thm}[Hermite Interpolation Formula] $(\mathop{\rm cf.~}$\cite{Spitzbart}$)$.
Let $f\in C^r\left[a,b\right]$ and $\xi,\eta\in[a,b]$. Then the Hermite interpolation polynomial $H_{2r+1}(x)$ of degree $2r+1$ such that
$H_{2r+1}^{(j)}(y)=f^{(j)}(y)$, for $y=\xi,\eta$, $0\leq j\leq r$, is given by
\begin{eqnarray}
H_{2r+1}(\xi,\eta,f;x)=\ds\sum\limits_{k=0}^{r}A_{0k}(x)f^{(k)}(\xi)+\ds\sum\limits_{k=0}^{r}A_{1k}(x)f^{(k)}(\eta),
\end{eqnarray}
where
\begin{eqnarray*}
A_{0k}(x)&=&(x-\eta)^{r+1}\dfrac{(x-\xi)^k}{k!}\ds\sum\limits_{s=0}^{r-k}\dfrac{1}{s!}g_0^{(s)}(\xi)(x-\xi)^s,\\
A_{1k}(x)&=&(x-\xi)^{r+1}\dfrac{(x-\eta)^k}{k!}\ds\sum\limits_{s=0}^{r-k}\dfrac{1}{s!}g_1^{(s)}(\eta)(x-\eta)^s,\\
g_0(x)&=&(x-\eta)^{-(r+1)},\\
g_1(x)&=&(x-\xi)^{-(r+1)}.
\end{eqnarray*}
\end{thm}
\section{Main Results}
Let $\mathcal{A}^r$ denote the class of $r$-times differentiable functions $\phi$ defined on $\mathbb{R}$ satisfying the following conditions:
\begin{itemize}
 \item[$(i)$] for some $\epsilon>0$, $\phi^{(s)}(x)=\mathcal{O}(|x|^{-0.5-\epsilon})$ as $x\to\pm\infty$, $s=0,1,\dots,r$.\\
\item [$(ii)$]$\mathop{\rm{esssup}}\limits_{w\in[0,1]}\ds\sum\limits_{l\in\mathbb{Z}}(w+l)^{2s}|\widehat{\phi}(w+l)|^2<\infty$,  $s=1,2,\dots,r$.
\end{itemize}
Clearly if $\phi$ is a continuously $r$- times differentiable function with compact support, then $\phi\in\mathcal{A}^r$. It is easy to prove that $\mathcal{A}^r\subseteq W(C,\ell^1)$,
which implies that $V(\phi)$ is a reproducing kernel Hilbert space.
We now assume that $\{T_n\phi:n\in\mathbb{Z}\}$ forms a Riesz basis for $V(\phi)$.  For $\phi\in\mathcal{A}^r$,  define
\begin{equation}
B_{s}(w):=\dfrac{\sum\limits_{l\in\mathbb{Z}}(w+l)^{2s}|\widehat{\phi}(w+l)|^2}{\sum\limits_{l\in\mathbb{Z}}|\widehat{\phi}(w+l)|^2},~w\in\mathbb{R},~s=1,2,\dots,r.
\end{equation}
Clearly $B_{s}(w)$ is a periodic function with period $1$.
If $M_{s}:=\mathop{\rm{esssup}}\limits_{w\in[0,1]}B_{s}(w)$, then
$$M_s\geq\dfrac{\sum\limits_{l\in\mathbb{Z}}(1/2+l)^{2s}|\widehat{\phi}(w+l)|^2}{\sum\limits_{l\in\mathbb{Z}}|\widehat{\phi}(w+l)|^2}\geq\dfrac{1}{4^s}.$$
\begin{thm}\label{pap4thm2.3}
Let $\phi\in\mathcal{A}^r$ be such that $\{T_{n}\phi:n\in\mathbb{Z}\}$ forms a Riesz basis for $V(\phi)$.
Then we have the Bernstein-type inequality
\begin{equation}\label{pap7eqn3.2}
\|f^{(s)}\|_2\leq(2\pi)^{s}\sqrt{M_{s}}\|f\|_2,~~s=1,2,\dots,r,
\end{equation}
for every $f\in V(\phi)$. Moreover, the constant $M_s$ depending on $\phi$ is sharp.
\end{thm}
\begin{pf}
Let $f\in V(\phi)$. Then $f(x)=\ds\sum\limits_{k\in\mathbb{Z}}c_k\phi(x-k)$, $(c_k)\in\ell^2(\mathbb{Z})$. We can easily prove that
$$f^{(s)}(x)=\ds\sum\limits_{k\in\mathbb{Z}}c_k\phi^{(s)}(x-k), ~~ s=1,2,\dots,r,$$
using the property $(i)$ of $\mathcal{A}^r$.
If $m_f(w)=\ds\sum\limits_{k\in\mathbb{Z}}c_ke^{-2\pi ikw}$,
then it follows from Plancherel identity that
\begin{eqnarray}
\|f^{(s)}\|_2^2&=&\|\widehat{f^{(s)}}\|^2_2=\big\|\sum\limits_{k\in\mathbb{Z}}c_k\widehat{\phi^{(s)}}(\cdot-k)\big\|^2_2\nonumber\\
&=&\int\limits_{-\infty}^{\infty}\Big|\sum\limits_{k\in\mathbb{Z}}(2\pi i)^{s}c_kw^s\widehat{\phi}(w)e^{-2\pi ikw}\Big|^2~\mathrm{d}w\nonumber\\
&=&(2\pi)^{2s}\int\limits_{0}^{1}|m_f(w)|^2\sum\limits_{l\in\mathbb{Z}}(w+l)^{2s}|\widehat{\phi}(w+l)|^2~\mathrm{d}w\nonumber\\
&=&(2\pi)^{2s}\int\limits_{0}^{1}B_{s}(w)|m_f(w)|^2\sum\limits_{l\in\mathbb{Z}}|\widehat{\phi}(w+l)|^2~\mathrm{d}w\nonumber\\
&\leq&(2\pi)^{2s}M_{s}\|f\|_2^2.\nonumber
\end{eqnarray}
The constant $M_s$ cannot be improved. In fact, the proof follows similar lines as in the proof of Theorem $1$ in \cite{BaZo}.
\end{pf}

Let $X=\{x_i:i\in\mathbb{Z}\}$, $\dots<x_{i-1}<x_i<x_{i+1}<\dots$, be a sampling set of density $\delta_\nu$.
If $X$ is separated by a constant $\gamma>0$, then
\begin{eqnarray}\label{pap7eqn3.3}
 x_{i+1}-x_i&=&x_{i+1}-x_i-x_{i+\nu}+x_{i+\nu}
 \leq\delta_\nu-[x_{i+\nu}-x_{i+1}]\nonumber\\
 &=&\delta_\nu-[x_{i+\nu}-x_{i+\nu-1}+x_{i+\nu-1}-x_{i+1}]\nonumber\\
 &\leq&\delta_\nu-\gamma-[x_{i+\nu-1}-x_{i+1}]\nonumber\\
 &\vdots&\nonumber\\
 &\leq&\delta_\nu-(\nu-1)\gamma.
\end{eqnarray}
We now define the weights
\begin{eqnarray}\label{pap7eqn3.4}
c_{i,l}&:=&\int\limits_{x_i}^{x_{i+1}}\dfrac{(x-x_{i+1})^{2l}}{l!^2}~\mathrm{d}x
=\dfrac{(x_{i+1}-x_i)^{2l+1}}{(2l+1)l!^2}\nonumber\\
&\leq&\dfrac{[\delta_\nu-(\nu-1)\gamma]^{2l+1}}{(2l+1)l!^2}
\leq[\delta_\nu-(\nu-1)\gamma]^{2l+1},~~i\in\mathbb{Z},~l\in \mathbb{N}\cup\{0\}.
\end{eqnarray}
Let $H_{2k-1}(x_i,x_{i+1},f;\cdot)$ denote the Hermite interpolation of $f$ in the interval $[x_i,x_{i+1}]$. Then 
it is proved in \cite{AntoRad2} that  
\begin{eqnarray}\label{pap7eqn3.5}
I_1:=\ds\int\limits_{x_i}^{x_{i+1}}|A_{0l}(x)|^2~\mathrm{d}x\leq c_{i,l}\left[\ds\sum\limits_{s=0}^{k-1}{k+s-1\choose s}\right]^2,
\end{eqnarray}
and
\begin{eqnarray}\label{pap7eqn3.6}
I_2:=\ds\int\limits_{x_i}^{x_{i+1}}|A_{1l}(x)|^2~\mathrm{d}x\leq c_{i,l}\left[\ds\sum\limits_{s=0}^{k-1}{k+s-1\choose s}\right]^2.
\end{eqnarray}
Consider  the operator $P:L^2(\mathbb{R})\to V(\phi)$ by
\begin{eqnarray}
(Pf)(x):=\langle f, \mathcal{K}_x\rangle,
\end{eqnarray}
where $\mathcal{K}_x(t)$ denotes the reproducing kernel of $V(\phi)$. Then $P$ is an orthogonal projection of $L^2(\mathbb{R})$ onto $V(\phi)$.
Now assume that $f$ and its first $k-1$ derivatives $f', \dots,f^{(k-1)}$ are sampled at a sequence $\{x_i:{i\in\mathbb{Z}}\}$.
We now introduce an approximation operator $T$ on $V(\phi)$ by
\begin{align*}
Tf:=P\left(\sum\limits_{i\in\mathbb{Z}}\left[\sum\limits_{j=1}^{\nu}H_{2k-1}(x_{\nu i+j-1},x_{\nu i+j},f;\cdot)\chi_{[x_{\nu i+j-1},x_{\nu i+j})}\right]\chi_{[x_{\nu i},x_{\nu i+\nu})}\right).
\end{align*}
Since $H_{2k-1}(\xi,\eta,\alpha f+g;x)=\alpha H_{2r+1}(\xi,\eta,f;x)+H_{2r+1}(\xi,\eta,g;x)$ for $\alpha\in\mathbb{C}$, the operator $T$ is linear.
Further,  
\begin{eqnarray}
\|T f\|_2^2
&\leq&\left\|\sum\limits_{i\in\mathbb{Z}}\left[\sum\limits_{j=1}^{\nu}H_{2k-1}(x_{\nu i+j-1},x_{\nu i+j},f;\cdot)\chi_{[x_{\nu i+j-1},x_{\nu i+j})}\right]\chi_{[x_{\nu i},x_{\nu i+\nu})}\right\|_2^2\nonumber\\
&=&\ds\int\limits_{\mathbb{R}}\left|\sum\limits_{i\in\mathbb{Z}}\left[\sum\limits_{j=1}^{\nu}H_{2k-1}(x_{\nu i+j-1},x_{\nu i+j},f;x)\chi_{[x_{\nu i+j-1},x_{\nu i+j})}(x)\right]\chi_{[x_{\nu i},x_{\nu i+\nu})}(x)\right|^2~\mathrm{d}x.\nonumber
\end{eqnarray}
Since the characteristic functions $\chi_{[x_{\nu i},x_{\nu i+\nu})}$, $i\in\mathbb{Z}$, have mutually disjoint support, it can be easily shown that
\begin{eqnarray}\label{pap7eqn3.8}
\|T f\|_2^2
&\leq&\ds\sum\limits_{i\in\mathbb{Z}}\int\limits_{x_{\nu i}}^{x_{\nu i+\nu}}\left|\sum\limits_{j=1}^{\nu}H_{2k-1}(x_{\nu i+j-1},x_{\nu i+j},f;x)\chi_{[x_{\nu i+j-1},x_{\nu i+j})}(x)\right|^2~\mathrm{d}x.
\end{eqnarray}
Since $x_{\nu i}<x_{\nu i+1}<\cdots<x_{\nu  i+\nu}$ and the characteristic functions $\chi_{[x_{\nu i+j-1},x_{\nu i+j})}$, $j=1,2,\dots,\nu$, have mutually disjoint support, \eqref{pap7eqn3.8} becomes
\begin{eqnarray}\label{pap7eqn3.9}
\|T f\|_2^2
&\leq&\ds\sum\limits_{i\in\mathbb{Z}}\sum\limits_{l=1}^\nu\int\limits_{x_{\nu i+l-1}}^{x_{\nu i+l}}\left|\sum\limits_{j=1}^{\nu}H_{2k-1}(x_{\nu i+j-1},x_{\nu i+j},f;x)\chi_{[x_{\nu i+j-1},x_{\nu i+j})}(x)\right|^2~\mathrm{d}x\nonumber\\
&=&\ds\sum\limits_{i\in\mathbb{Z}}\sum\limits_{j=1}^{\nu}\int\limits_{x_{\nu i+j-1}}^{x_{\nu i+j}}\big|H_{2k-1}(x_{\nu i+j-1},x_{\nu i+j},f;x)\big|^2~\mathrm{d}x\nonumber\\
&=&\ds\sum\limits_{j=1}^{\nu}\sum\limits_{i\in\mathbb{Z}}\int\limits_{x_{\nu i+j-1}}^{x_{\nu i+j}}\big|H_{2k-1}(x_{\nu i+j-1},x_{\nu i+j},f;x)\big|^2~\mathrm{d}x\nonumber\\
&=&\ds\sum\limits_{j=1}^{\nu}\ds\sum\limits_{i\in\mathbb{Z}}\int\limits_{x_{\nu i+j-1}}^{x_{\nu i+j}}
\left|\ds\sum\limits_{l=0}^{k-1}A_{0l}(x)f^{(l)}(x_{\nu i+j-1})+\ds\sum\limits_{l=0}^{k-1}A_{1l}(x)f^{(l)}(x_{\nu i+j})\right|^2~\mathrm{d}x\nonumber\\
&\leq&2\ds\sum\limits_{j=1}^{\nu}\left\{\ds\sum\limits_{i\in\mathbb{Z}}\int\limits_{x_{\nu i+j-1}}^{x_{\nu i+j}}
\left|\ds\sum\limits_{l=0}^{k-1}A_{0l}(x)f^{(l)}(x_{\nu i+j-1})\right|^2+\left|\ds\sum\limits_{l=0}^{k-1}A_{1l}(x)f^{(l)}(x_{\nu i+j})\right|^2~\mathrm{d}x\right\}\nonumber\\
&\leq&2k\sum\limits_{j=1}^{\nu}\left\{\ds\sum\limits_{i\in\mathbb{Z}}\int\limits_{x_{\nu i+j-1}}^{x_{\nu i+j}}
\ds\sum\limits_{l=0}^{k-1}|A_{0l}(x)f^{(l)}(x_{\nu i+j-1})|^2+\ds\sum\limits_{l=0}^{k-1}|A_{1l}(x)f^{(l)}(x_{\nu i+j})|^2\mathrm{d}x\right\}\nonumber\\
&=&2k\sum\limits_{j=1}^{\nu}\ds\sum\limits_{i\in\mathbb{Z}}
\ds\sum\limits_{l=0}^{k-1}|f^{(l)}(x_{\nu i+j-1})|^2\int\limits_{x_{\nu i+j-1}}^{x_{\nu i+j}}|A_{0l}(x)|^2\mathrm{d}x\nonumber\\
&&\hspace*{3cm}+2k\sum\limits_{j=1}^{\nu}\ds\sum\limits_{i\in\mathbb{Z}}\ds\sum\limits_{l=0}^{k-1}|f^{(l)}(x_{\nu i+j})|^2\int\limits_{x_{\nu i+j-1}}^{x_{\nu i+j}}|A_{1l}(x)|^2\mathrm{d}x\nonumber\\
&\leq&2kC(k)\sum\limits_{j=1}^{\nu}\left\{\ds\sum\limits_{i\in\mathbb{Z}}
\ds\sum\limits_{l=0}^{k-1}|f^{(l)}(x_{\nu i+j-1})|^2 c_{\nu i+j}
+\ds\sum\limits_{i\in\mathbb{Z}}\ds\sum\limits_{l=0}^{k-1}|f^{(l)}(x_{\nu i+j})|^2 c_{\nu i+j}\right\},\nonumber\\
\end{eqnarray}
using inequalities \eqref{pap7eqn3.5} and \eqref{pap7eqn3.6}, where $C(k)=\left[\ds\sum\limits_{s=0}^{k-1}{k+s-1\choose s}\right]^2$. Notice that it follows from \eqref{pap7eqn3.4} that
\begin{eqnarray}
\sum\limits_{j=1}^{\nu}\ds\sum\limits_{i\in\mathbb{Z}}
\ds\sum\limits_{l=0}^{k-1}|f^{(l)}(x_{\nu i+j-1})|^2 c_{\nu i+j}&=&\ds\sum\limits_{l=0}^{k-1}\sum\limits_{j=1}^{\nu}\ds\sum\limits_{i\in\mathbb{Z}}
|f^{(l)}(x_{\nu i+j-1})|^2 c_{\nu i+j}\nonumber\\
&\leq&\ds\sum\limits_{l=0}^{k-1}[\delta_\nu-(\nu-1)\gamma]^{2l+1}\sum\limits_{j=1}^{\nu}\ds\sum\limits_{i\in\mathbb{Z}}
|f^{(l)}(x_{\nu i+j-1})|^2 \nonumber\\
&=&C_{\nu,k}\ds\sum\limits_{l=0}^{k-1}\ds\sum\limits_{i\in\mathbb{Z}}
|f^{(l)}(x_{i})|^2,
\end{eqnarray}
where
$C_{\nu,k}=\max\{[\delta_\nu-(\nu-1)\gamma]^{2l+1}:l=0,\dots,k-1\}$. Similarly, we can show that
\begin{eqnarray}
\sum\limits_{j=1}^{\nu}\ds\sum\limits_{i\in\mathbb{Z}}
\ds\sum\limits_{l=0}^{k-1}|f^{(l)}(x_{\nu i+j})|^2 c_{\nu i+j}
&\leq&\ds\sum\limits_{l=0}^{k-1}[\delta_\nu-(\nu-1)\gamma]^{2l+1}\sum\limits_{j=1}^{\nu}\ds\sum\limits_{i\in\mathbb{Z}}
|f^{(l)}(x_{\nu i+j})|^2 \nonumber\\
&\leq&2C_{\nu,k}\ds\sum\limits_{l=0}^{k-1}\ds\sum\limits_{i\in\mathbb{Z}}
|f^{(l)}(x_{i})|^2.
\end{eqnarray}
Therefore, \eqref{pap7eqn3.9}  becomes
\begin{eqnarray}
\|T f\|_2^2\leq6kC(k)C_{\nu,k}\ds\sum\limits_{l=0}^{k-1}\ds\sum\limits_{i\in\mathbb{Z}}|f^{(l)}(x_{i})|^2.
\end{eqnarray}
Since $X$ is separated by the constant $\gamma>0$, there are at most $[1/\gamma]+1$ sampling points in every interval $[i,i+1]$. Consequently,  it follows from \cite{AlGr1}  that 
\begin{eqnarray}\label{pap7eqn3.13}
\|T f\|_2^2&\leq&6kC(k)C_{\nu,k}\ds\sum\limits_{l=0}^{k-1}\ds\sum\limits_{i\in\mathbb{Z}}|f^{(l)}(x_{i})|^2\\
&\leq&6kC(k)C_{\nu,k}(\gamma^{-1}+1)\sum\limits_{l=0}^{k-1}\|f^{l}\|_W^2\nonumber\\
&\leq&6kC(k)C_{\nu,k}(\gamma^{-1}+1)\|G_\phi\|_0^{-2}(\gamma^{-1}+1)\|f\|_2^2\label{pap7eqn3.14}\\
&=&C'_{\nu,k}\|f\|_2^2,\nonumber
\end{eqnarray}
where $C'_{\nu,k}=6kC(k)C_{\nu,k}(\gamma^{-1}+1)\|G_\phi\|_0^{-2}(\gamma^{-1}+1)>0$. Therefore $T$ is a bounded linear operator on $V(\phi)$. We now prove our main result.
\begin{thm}\label{pap3thm2.2}
Let $\phi\in\mathcal{A}^{2k}$ be such that $\{T_{n}\phi:n\in\mathbb{Z}\}$ forms a Riesz basis for $V(\phi)$. Assume that $\{x_i:{i\in\mathbb{Z}}\}$ is separated by a constant $\gamma>0$.
If there exists a positive  integer $\nu$ such that 
$$\sup\limits_i(x_{i+\nu}-x_i)=\delta_\nu<\dfrac{\nu}{2\pi}\left(\dfrac{c_{k}^2}{M_{2k}}\right)^{\frac{1}{4k}},$$ then 
any $f\in V(\phi)$ can be reconstructed from the sample values $\{f^{(j)}(x_i):j=0,1,\dots,k-1,i\in\mathbb{Z}\}$ using the following iteration algorithm.
Set
\begin{eqnarray*}
f_0&=&Tf=P\left(\sum\limits_{i\in\mathbb{Z}}\left[\sum\limits_{j=1}^{\nu}H_{2k-1}(x_{\nu i+j-1},x_{\nu i+j},f;\cdot)\chi_{[x_{\nu i+j-1},x_{\nu i+j})}\right]\chi_{[x_{\nu i},x_{\nu i+\nu})}\right),\\
f_{n+1}&=&f_n+T(f-f_n),~ n\geq 0,
\end{eqnarray*}
where $H_{2k-1}(x_i,x_{i+1},f;\cdot)$ denotes the Hermite interpolation of $f$ in the interval $[x_i,x_{i+1}]$.
Then we have $\lim\limits_{n\to\infty}f_n=f$. The error estimate after $n$ iterations becomes
\begin{eqnarray*}
\|f-f_n\|_2&\leq&\left(\dfrac{[\delta_\nu-(\nu-1) \gamma]^{2k}}{c_{k}}(2\pi)^{2k}\sqrt{M_{2k}}\right)^{n+1}\|f\|_2.
\end{eqnarray*}
\end{thm}
\begin{proof}
Since  $f=Pf=P\left(\ds\sum\limits_{i\in\mathbb{Z}}f\chi_{[x_{\nu i},x_{\nu i+\nu})}\right)$, and the characteristic functions $\chi_{[x_{\nu i},x_{\nu i+\nu})}$, $i\in\mathbb{Z}$, have mutually disjoint support,
it can be easily shown that
\begin{eqnarray}\label{pap7eqn3.15}
\|f-T f\|_2^2
&=&\ds\sum\limits_{i\in\mathbb{Z}}\int\limits_{x_{\nu i}}^{x_{\nu i+\nu}}\left|\left[f-\sum\limits_{j=1}^{\nu}H_{2k-1}(x_{\nu i+j-1},x_{\nu i+l},f;x)\right]\chi_{[x_{\nu i+j-1},x_{\nu i+j})}(x)\right|^2~\mathrm{d}x\nonumber\\
&=&\ds\sum\limits_{i\in\mathbb{Z}}\sum\limits_{j=1}^\nu\int\limits_{x_{\nu i+j-1}}^{x_{\nu i+j}}\left|\left[f-\sum\limits_{j=1}^{\nu}H_{2k-1}(x_{\nu i+j-1},x_{\nu i+j},f;x)\right]\chi_{[x_{\nu i+j-1},x_{\nu i+j})}(x)\right|^2~\mathrm{d}x\nonumber\\
&=&\ds\sum\limits_{i\in\mathbb{Z}}\sum\limits_{j=1}^{\nu}\int\limits_{x_{\nu i+j-1}}^{x_{\nu i+j}}\big|f(x)-H_{2k-1}(x_{\nu i+j-1},x_{\nu i+j},f;x)\big|^2~\mathrm{d}x.
\end{eqnarray}
Now applying Wirtinger-Sobolev inequality with the bound  given in \eqref{pap7eqn2.6a}, \eqref{pap7eqn3.15} becomes
\begin{eqnarray}
\|f-Tf\|_2^2
&\leq&\ds\sum\limits_{i\in\mathbb{Z}}\sum\limits_{j=1}^{\nu}\dfrac{(x_{\nu i+j}-x_{\nu i+j-1})^{4k}}{c_{k}^2}
\int\limits_{x_{\nu i+j-1}}^{x_{\nu i+j}}|f^{(2k)}(x)|^2~\mathrm{d}x.
\end{eqnarray}
Since $x_{i+1}-x_i\leq\delta_\nu-(\nu-1)\gamma$ for every $i\in\mathbb{Z}$, we get
\begin{eqnarray}\label{pap7eqn3.17}
\|f-Tf\|_2^2
&\leq&\ds\sum\limits_{i\in\mathbb{Z}}\dfrac{(\delta_\nu-(\nu-1)\gamma)^{4k}}{c_{k}^2}
\int\limits_{x_{\nu i}}^{x_{\nu i+\nu}}|f^{(2k)}(x)|^2~\mathrm{d}x\nonumber\\
&=&\dfrac{[\delta_\nu-(\nu-1)\gamma]^{4k}}{c_{k}^2}\|f^{(2k)}\|_2^2\nonumber\\
&\leq&\dfrac{[(\nu-1) \gamma-\delta_\nu]^{4k}}{c_{k}^2}(2\pi)^{4k}M_{2k}\|f\|_2^2,
\end{eqnarray}
using Bernstein's inequality \eqref{pap7eqn3.2}.
Since $\sup\limits_i(x_{i+\nu}-x_i)=\delta_\nu$ and $\inf\limits_{i\neq j} |x_i-x_j|=\gamma$, we have $\gamma\leq\dfrac{\delta_\nu}{\nu}$. Hence, \eqref{pap7eqn3.17} becomes
\begin{eqnarray}
\|f-Tf\|_2^2
&\leq&\dfrac{\left[\delta_\nu/\nu\right]^{4k}}{c_{k}^2}(2\pi)^{4k}M_{2k}\|f\|_2^2,
\end{eqnarray}
If $\delta_\nu<\dfrac{\nu}{2\pi}\left(\dfrac{c_{k}^2}{M_{2k}}\right)^{\frac{1}{4k}}$, then $\dfrac{\left(\delta_\nu/\nu\right)^{4k}}{c_{k}^2}(2\pi)^{4k}M_{2k}<1$. Therefore, the operator $T$ is invertible on $V(\phi)$.
Consequently, the proposed reconstruction algorthim follows from Proposition \ref{pap2prop2.1} and  we obtain the error estimate from \eqref{pap7eqn3.17}.
\end{proof}
\begin{cor}\label{pap7cor3.1}
Let $\phi\in\mathcal{A}^{2k}$ be such that $\{T_{n}\phi:n\in\mathbb{Z}\}$ forms a Riesz basis for $V(\phi)$. Assume that $\{x_i:{i\in\mathbb{Z}}\}$ is separated by a constant $\gamma>0$.
If 
$$\sup\limits_i(x_{i+\nu}-x_i)=\delta_\nu<\dfrac{\nu}{2\pi}\left(\dfrac{c_{k}^2}{M_{2k}}\right)^{\frac{1}{4k}},~~\text{for some~}\nu\in\mathbb{N,}$$ 
then there exist positive constants 
$A_{\nu,k}$  and $B_{\nu,k}$ such that
\begin{eqnarray}\label{pap3eqn2.7}
A_{\nu,k}\| f\|_2^2&\leq&\ds\sum\limits_{i\in\mathbb{Z}}
\ds\sum\limits_{l=0}^{k-1}|f^{(l)}(x_i)|^2\leq B_{\nu,k}\| f\|_2^2,
\end{eqnarray}
every $f\in V(\phi)$.
\end{cor}
\begin{proof}
The existence of $B_{\nu,k}$ follows from \eqref{pap7eqn3.13} and \eqref{pap7eqn3.14}.
Recall that $$Tf=P\left(\sum\limits_{i\in\mathbb{Z}}\left[\sum\limits_{j=1}^{\nu}H_{2k-1}(x_{\nu i+j-1},x_{\nu i+j},f;\cdot)\chi_{[x_{\nu i+j-1},x_{\nu i+j})}\right]\chi_{[x_{\nu i},x_{\nu i+\nu})}\right).$$ 
We now estimate the lower bound of $\ds\sum\limits_{i\in\mathbb{Z}}\ds\sum\limits_{l=0}^{k-1}|f^{(l)}(x_i)|^2$. Since the operator $T$ is invertible on $V(\phi)$, we have
\begin{eqnarray}\label{pap3eqn2.8}
\|f\|_2^2&=&\|T^{-1}Tf\|_2^2\nonumber\\
&\leq&\|T^{-1}\|^2\|Tf\|_2^2\nonumber\\
&\leq&(1-\|I-T\|)^{-2}\|Tf\|_2^2\nonumber\\
&\leq&\left(1-\dfrac{[\delta_\nu-(\nu-1) \gamma]^{2k}}{c_{k}}(2\pi)^{2k}\sqrt{M_{2k}}\right)^{-2}\|Tf\|_2^2\nonumber\\
&\leq&\left(1-\dfrac{[\delta_\nu-(\nu-1) \gamma]^{2k}}{c_{k}}(2\pi)^{2k}\sqrt{M_{2k}}\right)^{-2}6kC(k)C_{\nu,k}\ds\sum\limits_{l=0}^{k-1}\ds\sum\limits_{i\in\mathbb{Z}}|f^{(l)}(x_{i})|^2,\nonumber
\end{eqnarray}
using inequality \eqref{pap7eqn3.13}, from which the existence of $A_{\nu,k}$ follows.
\end{proof}

\subsection{Sampling Density in the Space of Bandlimited Functions}
The classical Bernstein inequality states that for every $f\in \mathcal{B}_\sigma$,
\begin{eqnarray}
\|f^{(k)}\|_2\leq \sigma^k\|f\|_2.
\end{eqnarray}
It is clear that the space $\mathcal{B}_\pi$ coincides with the shift-invariant space $V(\phi)$ generated by the function $\phi(x)=\dfrac{\sin\pi x}{\pi x}$.
If $f\in \mathcal{B}_\sigma$, then the function $g(x)=f\left(\dfrac{\pi x}{\sigma}\right)\in \mathcal{B}_\pi$. Consequently, we obtain the following  result from Corollary \ref{pap7cor3.1}.
\begin{thm}\label{pap7thm3.3}
Let $\{x_i:{i\in\mathbb{Z}}\}$ be separated by a constant $\gamma>0$.
If  
$$\sup\limits_i(x_{i+\nu}-x_i)=\delta_\nu<\dfrac{\nu}{\sigma}c_k^{1/2k},~ \text{for some~} \nu\in\mathbb{N},$$
then  there exist positive constants 
$A_{\nu,k}$  and $B_{\nu,k}$ such that
\begin{eqnarray}\label{pap3eqn2.7}
A_{\nu,k}\| f\|_2^2&\leq&\ds\sum\limits_{i\in\mathbb{Z}}
\ds\sum\limits_{l=0}^{k-1}|f^{(l)}(x_i)|^2\leq B_{\nu,k}\| f\|_2^2,
\end{eqnarray}
for every $f\in \mathcal{B}_\sigma$.
\end{thm}
Using Bernstein's inequality and Wirtinger-Sobolev constants $c_k$, we now provide maximum gap value for certain values of $\nu$ and $k$ with $\sigma=\pi$ in Table 1.

\begin{center}
 \begin{tabular}{|l|l|l|l|l|l|l|}
  \hline
  \multicolumn{5}{|c|}{$\delta_\nu<\nu c_k^{1/2k}$}\\
  \hline
 \backslashbox{$k$}{$\nu$}  &1&2&3&4\\
  \hline
  1 &1&2&3&4\\
   \hline
  2 &1.5006&3.0112&4.5018&6.0024\\
  \hline
  3 &2&4&6&8\\
  \hline
  4 &1.9169&3.8338&5.7507&7.6676\\
  \hline
  5 &2.361&4.722&7.083&9.444\\
  \hline
  6 &2.8094&5.6188&8.4282&11.2376\\
  \hline
   7 &3.2608&6.5216&9.7824&13.0432\\
   \hline
    8 &3.7144&7.4288&11.1432&14.8576\\
  \hline
   9 &4.1697&8.3394&12.5091&16.6788\\
  \hline
   10 &4.6263&9.2526&13.8789&18.5052\\
  \hline
   20 &10.0044&20.0088&30.0132&40.0176\\
  \hline
   40 &19.5623&39.1246&58.6869&78.2492\\
  \hline
  \end{tabular}\\
\vspace*{0.3cm}
Table 1
\end{center}
Consider the sampling set $X$ whose elements are chosen in this way:
\begin{equation*}
x_i= \left\{\begin{array}{cc}
i-0.25&\mbox{~~if $i>0$,}\\
0&\mbox{~~if $i=0$,}\\
i+0.25&\mbox{~~if $i<0$.}\\
\end{array} \right.
\end{equation*}
It is well known that the above sampling set $X$ is not a stable set of sampling for $\mathcal{B}_\pi$ (see \cite{Young}).  But $\sup\limits_i(x_{i+\nu}-x_i)=\nu$, for every $\nu\in\mathbb{N}$.
Therefore, the maximum gap condition in the previous theorem is 
sharp when $k=1$. Thus we have proved Theorem \ref{pap7thm1.1} for $\mathcal{B}_\pi$.

Recall that the windowed Fourier transform of $f\in L^2(\mathbb{R})$  is defined as
\begin{eqnarray*}\label{STFT}
F_gf(x,y)=\int\limits_{-\infty}^{\infty}f(t)\overline{g(t-x)}e^{-2\pi ity}~\mathrm{d}t=\langle f, M_yT_xg \rangle,
\end{eqnarray*}
where $M_y$ is the modulation operator defined by $$M_yf(x):=e^{2\pi iyx}f(x),~~f\in L^2(\mathbb{R}).$$
We now construct a Gabor-type frame which is a generalized version of Theorem $3$ in \cite{Grochenig2} or Theorem 8.29 in \cite{FeiGr}.

\begin{thm}\label{pap2thm5.2}
Suppose that $g\in \mathcal{B}_\sigma\bigcap L^1(\mathbb{R})$ and there exist two positive constants $a,b$ such that the sequence $\{y_j\}$ satisfies
\begin{eqnarray}\label{pap4eqn2.3}
a\leq\sum\limits_{j\in\mathbb{Z}}|\widehat{g}(\xi-y_j)|^2\leq b.
\end{eqnarray}
If there exists positive integer $\nu$ such that 
$$\sup\limits_i(x_{i+\nu,j}-x_{i,j})=\delta_\nu<\dfrac{\nu}{\sigma}c_k^{1/2k},$$ 
then the collection
\begin{eqnarray*}
\Big\{M_{y_j}T_{x_{i,j}}g^{(l)}:l=0, 1, \dots,k-1, i, j\in\mathbb{Z}\Big\}
\end{eqnarray*}
forms a frame for $L^2(\mathbb{R})$.
\end{thm}
\begin{pf}
Let $\widetilde{g}(t)=\overline{g(-t)}$. Then for each $y\in\mathbb{R}$, $F_gf(x,y)=M_{-y}f*\widetilde{g}(x)$ is a bandlimited function with
$\text{supp}(\widehat{M_{-y}f*\widetilde{g})}\subseteq\text{supp}(\widehat{g})=[-\sigma,\sigma]$, \textit{i.e.}, $F_g(\cdot,y)\in B_\sigma$.
Since $\dfrac{\partial^l}{\partial x^l}F_gf(x,y)=(-1)^lF_{g^{(l)}}f(x,y)$, by Corollary \ref{pap7cor3.1}, there exist 
$A_{\nu,k}$  $B_{\nu,k}>0$ such that
\begin{eqnarray}\label{pap4eqn2.4}
A_{k,\nu}\|M_{-y_j}f*\widetilde{g}\|^2_2&\leq&\ds\sum\limits_{i\in\mathbb{Z}}
\ds\sum\limits_{l=0}^{k-1}|F_{g^{(l)}}f(x_{i,j},y_j)|^2\leq B_{k,\nu}\|M_{-y_j}f*\widetilde{g}\|^2_2.\nonumber\\
\end{eqnarray}
Using Plancherel's identity, we get $\|M_{-y_j}f*\widetilde{g}\|^2_2=\ds\int\limits_{\mathbb{R}}|\widehat{f}(\xi)|^2|\widehat{g}(\xi-y_j)|^2\mathrm{d}\xi$.
Summing over all $j\in\mathbb{Z}$, we get
\begin{eqnarray}
\sum\limits_{j\in\mathbb{Z}}\|M_{-y_j}f*\widetilde{g}\|^2_2=\int\limits_{\mathbb{R}}|\widehat{f}(\xi)|^2
\sum\limits_{j\in\mathbb{Z}}|\widehat{g}(\xi-y_j)|^2\mathrm{d}\xi.
\end{eqnarray}
Hence, it follows from  \eqref{pap4eqn2.3} and \eqref{pap4eqn2.4} that
\begin{eqnarray*}
aA_{k,\nu}\|f\|^2_2&\leq&\ds\sum\limits_{i,j\in\mathbb{Z}}
\ds\sum\limits_{l=0}^{k-1}|F_{g^{(l)}}f(x_{i,j},y_j)|^2\leq bB_{k,\nu}\|f\|^2_2.
\end{eqnarray*}
Since $F_gf(x,y)=\langle f, M_yT_xg \rangle$, we get the required assertion.
\end{pf}

\subsection{Sampling Density in the Shift-invariant Spline Spaces}
The $B$-splines $Q_m$ are defined inductively as follows:
\begin{align*}
\hspace*{-1cm}Q_1(x):=\chi_{[0,1)}(x),
\end{align*}
\begin{align*}
Q_{m+1}(x)=
(Q_m*Q_1)(x):=\displaystyle{\int\limits_{-\infty}^{\infty}Q_m(x-t)Q_1(t)\mathrm{d}t},
\hspace*{0.5cm}m\geq1.
\end{align*}
It is well known that $f\in V(Q_m)$ if and only if $f\in C^{m-2}(\mathbb{R})\cap L^2(\mathbb{R})$ and the restriction of $f$ to each interval $\left[n,n+1\right),n\in\mathbb{Z}$ is a polynomial of degree $m-1$.
The following result is due to Babenko and Zontov in \cite{BaZo}.
\begin{thm}[Bernstein's inequality]\label{pap2thmBernineq}
For any $f\in V(Q_m)$,  we have
\begin{eqnarray}\label{pap7eqn3.24}
\|f^{(k)}\|_{L^2(\mathbb{R})}\leq \pi^k \sqrt{\dfrac{K_{2(m-k)-1}}{K_{2m-1}}}~\|f\|_{L^2(\mathbb{R})},~~~ k=1,2,3,\dots,m-1,
\end{eqnarray}
where $K_m$'s are the Krein-Favard constants:
\begin{eqnarray*}
K_m=\dfrac{4}{\pi}\sum\limits_{\nu=0}^{\infty}\frac{(-1)^{\nu(m+1)}}{(2\nu+1)^{m+1}},~~~m=0,1,2,3,\dots
\end{eqnarray*}
Moreover, the inequality \eqref{pap7eqn3.24} is sharp.
\end{thm}
The first six values of $K_m$ are the following:
\begin{align*}
K_0=1,~K_1=\dfrac{\pi}{2},~ K_2=\dfrac{\pi^2}{8}, ~K_3=\dfrac{\pi^3}{24},~ K_4=\dfrac{5\pi^4}{384}, ~ K_5=\dfrac{\pi^5}{240},
\end{align*}
and $\lim\limits_{m\to\infty}K_m=\dfrac{4}{\pi}$. The constants $K_m$ also satisfy the following inequalities:
$$1=K_0<K_2<K_4<\cdots<\dfrac{4}{\pi}<\cdots<K_5<K_3<K_1=\dfrac{\pi}{2}.$$
We refer to \cite{GoFe} in this connection. Since the inequality \eqref{pap7eqn3.24} is sharp,  $M_{2k}=\dfrac{K_{2(m-2k)-1}}{2^{4k}K_{2m-1}}$. Hence, we obtain the following result from Corollary \ref{pap7cor3.1}.
\begin{thm}
Let $\{x_i:{i\in\mathbb{Z}}\}$ be separated by a constant $\gamma>0$.
If  
$$\sup\limits_{i\in\mathbb{Z}}(x_{i+\nu}-x_i)=\delta_\nu<\dfrac{\nu}{\pi}\left(\dfrac{c_{k}^2K_{2m-1}}{K_{2(m-2k)-1}}\right)^{\frac{1}{4k}},~\text{for some}~\nu\in\mathbb{N},$$
then  there exist positive constants 
$A_{\nu,k}$  and $B_{\nu,k}$ such that
\begin{eqnarray}\label{pap3eqn2.7}
A_{\nu,k}\| f\|_2^2&\leq&\ds\sum\limits_{i\in\mathbb{Z}}
\ds\sum\limits_{l=0}^{k-1}|f^{(l)}(x_i)|^2\leq B_{\nu,k}\| f\|_2^2,
\end{eqnarray}
every $f\in V(Q_m)$, where $m\geq 2k+1$.
\end{thm}
The maximum gap condition is given in \cite{AntoRad1} for $\nu=1$. In Table 2, we provided maximum gap value for certain values of $\nu$ and $k$ with $m=8$.
In Table 3, we assume $m=10$.\\
\begin{center}
 \begin{tabular}{|l|l|l|l|l|l|l|}
  \hline
  \multicolumn{5}{|c|}{$\delta_\nu<\dfrac{\nu}{\pi}\left(\dfrac{c_{k}^2K_{15}}{K_{15-4k}}\right)^{\frac{1}{4k}}$}\\
  \hline
\backslashbox{$k$}{$\nu$}  &1&2&3&4\\
  \hline
  1 &0.9999&1.9998&2.9997&3.9996\\
   \hline
  2 &1.5055&3.011&4.5165&6.022\\
  \hline
  3 &1.9975&3.995&5.9925&7.99\\
  \hline
     \end{tabular}\\
\vspace*{0.3cm}
Table 2
\end{center}

\vspace*{0.5cm}
\begin{center}
 \begin{tabular}{|l|l|l|l|l|l|}
  \hline
  \multicolumn{5}{|c|}{$\delta_\nu<\dfrac{\nu}{\pi}\left(\dfrac{c_{k}^2K_{19}}{K_{19-4k}}\right)^{\frac{1}{4k}}$}\\
  \hline
 \backslashbox{$k$}{$\nu$} &1&2&3&4\\
  \hline
  1 &0.9999&1.9998&2.9997&3.9996\\
   \hline
  2 &1.5056&3.0112&4.5168&6.0224\\
  \hline
  3 &1.9999&3.9998&5.9997&7.996\\
  \hline
     \end{tabular}\\
\vspace*{0.3cm}
Table 3
\end{center}
\subsection{Sampling Density in Shift-invariant Space Generated by Meyer Scaling Function}
Consider a function $\vartheta(w)$ defined on the interval
$0\leq w \leq 1$ satisfying the following properties:
\begin{itemize}
\item[$(P_1)$] $0\leq\vartheta(w)\leq 1$,
\item[$(P_2)$] $\vartheta(w)+\vartheta(1-w)=1$,
\item[$(P_3)$] $\vartheta(w)$ is a monotonically decreasing function,
\item[$(P_4)$] $\vartheta(w)=1$, $0\leq w \leq \dfrac{1}{3}$,
\item[$(P_5)$] $\vartheta(w)\geq 2-3w$ in $[1/3,1/2]$ and $\vartheta(w)\leq 2-3w$ in $[1/2,2/3]$.
\end{itemize}
The function $\vartheta$ is extended to the real line by setting
$\vartheta(w)=\vartheta(-w)$ for $-1\leq w\leq 0$ and
$\vartheta(w)=0$ for $|w| > 1$. Then the Meyer scaling function
is defined as
\begin{eqnarray*}
\phi(x):=\displaystyle\int\limits_{-1}^{1}\sqrt{\vartheta(w)}e^{2\pi iw x}\mathrm{d}w.
\end{eqnarray*}
 Consider the function
\begin{equation*}
 \widehat{\phi}(w)= \left\{\begin{array}{cc}
1& \mbox{ $|w|\leq \dfrac{1}{3}$,}\\
\cos\Big[\dfrac{\pi}{2}\nu(3|w|-1)\Big] &\mbox{ $\dfrac{1}{3}\leq|w|\leq \dfrac{2}{3}$,}\\
0 &\mbox{otherwise,}
\end{array} \right.
\end{equation*}
where $\nu$ is a  $C^\infty$ function satisfying
\begin{equation*}
 \nu(x)= \left\{\begin{array}{cc}
0& \mbox{~~if~ $x\leq0$,}\\
1 &\mbox{~~if~ $x\geq1$,}
\end{array} \right.
\end{equation*}
with the additional property
\begin{eqnarray*}
\nu(x)+\nu(1-x)=1.
\end{eqnarray*}
It is proved in \cite{AntoRad3} that $\phi$ satisfies the properties $(P_1)$ to $(P_5)$. 
We refer to \cite{AntoRad3} for further details.

Since the collection $\{T_n\phi:n\in\mathbb{Z}\}$ forms an
orthonormal basis for $V(\phi)$, we get
$\sum\limits_{l\in\mathbb{Z}}|\widehat{\phi}(w+l)|^2=1 ~~a.e.,~w\in\mathbb{R}$.
Therefore, 
\begin{equation*}
B_s(w)= \left\{\begin{array}{cc}
w^{2s}& \mbox{~~if $0\leq w\leq 1/3$,}\\
(w-1)^{2s}\vartheta(w-1)+ w^{2s}\vartheta(w)&\mbox{~~if $1/3\leq w\leq 2/3$,}\\
(w-1)^{2s}&\mbox{~~if $2/3\leq w\leq 1$.}\\
\end{array} \right.
\end{equation*}
Using property $P_2$, we get
\begin{equation*}
B_s(w)= \left\{\begin{array}{cc}
w^{2s}& \mbox{~~if $0\leq w\leq 1/3$,}\\
(w-1)^{2s}+[w^{2s}-(w-1)^{2s}]\vartheta(w)&\mbox{~~if $1/3\leq w\leq 2/3$,}\\
(w-1)^{2s}&\mbox{~~if $2/3\leq w\leq 1$.}\\
\end{array} \right.
\end{equation*}
Since  $1/2\leq\vartheta(w)\leq1$ and $w^{2r}-(w-1)^{2r}\leq 0$ in the interval $[1/3,1/2]$, and $0\leq\vartheta(w)\leq1/2$ and $w^{2s}-(w-1)^{2s}\leq0$ in the interval $[1/2,2/3]$, we get
\begin{eqnarray*}
B_s(w)&\leq&(w-1)^{2s}+[w^{2s}-(w-1)^{2s}]\Big(\dfrac{1}{2}\Big)\\
&=&\dfrac{w^{2s}+(w-1)^{2s}}{2}\leq\dfrac{(1/3)^{2s}+(2/3)^{2s}}{2}.
\end{eqnarray*}
in $[1/3,2/3]$. In this case, $M_s\leq\dfrac{(1/3)^{2s}+(2/3)^{2s}}{2}$.
It is proved in \cite{AntoRad3} that $M_1=\dfrac{1}{4}$.
We now compute the value of $M_2$.
Using property $P_2$, we get
\begin{eqnarray*}
B_2(w)&=&(w-1)^{4}[1-\vartheta(w)]+ w^{4}\vartheta(w)\\
&=&w^4-4w^3+6w^2-4w+1+[4w^3-6w^2+4w-1]\vartheta(w),
\end{eqnarray*}
in the interval $[1/3, 2/3]$. Consider the function
$$h(w)=4w^3-6w^2+4w-1, ~w\in[1/3,2/3].$$
Then  $h(1/2)=0$ and $h^\prime(w)=12w^2-12w+4=12(w^2-w)+4$.
Since the function $w^2-w$ attains its minimum value at $w=1/2$, we have
$h^\prime(w)=12w^2-12w+4=12(w^2-w)+4\geq 1$. Therefore, $h(w)$ is an increasing function in the interval $[1/3, 2/3]$, which implies
that $h(w)$ is negative in $[1/3,1/2]$ and positive in $[1/2,2/3]$.
Now using property $P_5$, we get
\begin{eqnarray*}
B_2(w)&=&w^4-4w^3+6w^2-4w+1+[4w^3-6w^2+4w-1]\vartheta(w)\\
&\leq&w^4-4w^3+6w^2-4w+1+(4w^3-6w^2+4w-1)(2-3w)\\
&=&-11w^4+22w^3-18w^2+7w-1.
\end{eqnarray*}
Since the function
$g(w)=-11w^4+22w^3-18w^2+7w-1$ attains its maximum at $w=1/2$, we have $B_2(w)\leq \dfrac{1}{16}$ in $[1/3,2/3]$.
Therefore, $M_2=\dfrac{1}{16}$. Hence, it follows from Corollary \ref{pap7cor3.1} that if there exists a positive integer $\nu$ such that 
$$\sup\limits_i(x_{i+\nu}-x_i)=\delta_\nu<\nu,$$
then the set $X=\{x_i:i\in\mathbb{Z}\}$ is a stable set of sampling for $V(\phi)$.
Further, arguing as in \cite{AntoRad3}, we can prove that the maximum gap $\delta_\nu<\nu$ is sharp. 
Thus we have established  Theorem \ref{pap7thm1.1} for a shift-invariant space generated by a Meyer scaling function, which we formulate here as follows:
\begin{thm}\label{pap7thm3.7}
Let $\vartheta$  be a $C^\infty$ function satisfying the properties $(P_1)$ to $(P_5)$ and
\begin{eqnarray*}
\phi(x):=\displaystyle\int\limits_{-1}^{1}\sqrt{\vartheta(w)}e^{2\pi iw x}\mathrm{d}w
\end{eqnarray*}
be the Meyer scaling function. Assume that $\{x_i:{i\in\mathbb{Z}}\}$ is separated by a constant $\gamma>0$.
If there exists a positive integer $\nu$ with
$$\sup\limits_i(x_{i+\nu}-x_i)=\delta_\nu<\nu,$$
then there exist positive constants $A_\nu$ and $B_\nu$ such that
\begin{eqnarray}\label{pap3eqn2.7}
A_\nu\| f\|_2^2&\leq&\ds\sum\limits_{i\in\mathbb{Z}}|f(x_i)|^2\leq B_\nu\| f\|_2^2,
\end{eqnarray}
for every $f\in V(\phi)$.  Further, the maximum gap  $\delta_\nu<\nu$  is sharp. 
\end{thm}
 
\section{Open problems}
In this section, we discuss some open problems related to the maximum gap condition for shift-invariant spaces.
\begin{itemize}
 \item [(1)] We first observed that the proof of Theorem \ref{pap7thm1.1} in section 3 is based on one of the constants $M_{s}$,
\begin{equation}
M_{s}=\mathop{\rm{esssup}}\limits_{w\in[0,1]}\dfrac{\sum\limits_{l\in\mathbb{Z}}(w+l)^{2s}|\widehat{\phi}(w+l)|^2}{\sum\limits_{l\in\mathbb{Z}}|\widehat{\phi}(w+l)|^2},~w\in\mathbb{R},~s=1,2,\dots,r.
\end{equation}
which appears in Bernstein inequality on a shift-invariant space $V(\phi)$.
The value of the constant $M_s$ is known only for the function $\tfrac{\sin \pi x}{\pi x}$ and  the $B$-splines.
The constant $M_1$ is known for the Meyer scaling function and Littlewood-Paley wavelet (see \cite{AntoRad3}). 
Arguing as in \cite{AntoRad3}, we  can easily obtain other values of $M_s$ for Littlewood-Paley wavelet. In this paper, we found the value of $M_2$ for Meyer scaling function. 
We do not know the value of $M_s,~s>2$ for Meyer scaling function and this problem is still open. In general,  the problem of finding the value of $M_s$ for some reasonable function $\phi$ will be more interesting.
\end{itemize}
\begin{itemize}
\item [(2)]  We have proved the sharp maximum gap condition $\sup_{i}(x_{i+\nu}-x_i)$ for the space of entire function of exponential type $\pi$ and  shift-invariant space generated by Meyer scaling function.
We observed from the Table 2 and 3 that the maximum gap $\sup_{i}(x_{i+\nu}-x_i)$ is very close to $\nu$ for shift-invariant spline spaces. 
Aldroubi and Gr\"{o}chenig in \cite{AlGr1} proved that if $\{x_{i}\}$ is  a separated set with $\sup_{i}(x_{i+1}-x_i)<1$, 
then  $\{x_i:i\in\mathbb{Z}\}$ is a stable set of sampling for the shift-invariant spline space $V(Q_m)$. For other cases $\nu\geq1$, it seems natural to state the following 
\begin{conj}
If $\{x_i:i\in\mathbb{Z}\}$ is  a separated 
set with $$\sup\limits_{i}(x_{i+\nu}-x_i)<\nu,$$ 
then  $\{x_i:i\in\mathbb{Z}\}$ is a stable set of sampling for $V(Q_m)$.
\end{conj}
We believe that the above conjecture also holds for a much larger class of shift-invariant spaces.
If $f$ and its first $k-1$ derivatives $f', \dots,f^{(k-1)}$ are sampled at a sequence $\{x_i:i\in\mathbb{Z}\}$, then we can generalize Conjecture 1 as follows.
\begin{conj}
If $\{x_{i}:i\in\mathbb{Z}\}$ is  a separated set with $$\sup\limits_{i}(x_{i+\nu}-x_i)<k\nu,$$
then every function $f\in V(\phi)$ can be reconstructed stably from its nonuniform sample values  $\{f^{(j)}(x_i):j=0,\dots, k-1, i\in\mathbb{Z}\}$.
\end{conj}
\end{itemize}
\begin{itemize}
\item [(3)] In this paper, we did not discuss about  necessary condition for the stable set of sampling for shift-invariant spaces. It will be interesting to 
find the necessary condition for stable set of sampling for shift-invariant spaces involving the maximum density condition $\sup_{i}(x_{i+\nu}-x_i)$, $\nu\in\mathbb{N}$.
\end{itemize}

\section*{Ackowledgement}
The author  wish to thank  Sairam Kaliraj and N. Karimilla Bi  for reading the manuscript and useful discussion to improve the earlier version of the manuscript.

\end{document}